\documentclass{amsart}

\usepackage{latexsym, graphics, amscd, amssymb}

\newcommand{\cA}{\mathcal{A}}

\newcommand{\cC}{\mathcal{C}}
\newcommand{\D}{\mathcal{D}}
\newcommand{\cE}{\mathcal{E}}

\newcommand{\cG}{\mathcal{G}}
\newcommand{\cM}{\mathcal{M}}
\newcommand{\cN}{\mathcal{N}}
\newcommand{\LL}{\mathcal{L}}
\newcommand{\cO}{\mathcal{O}}

\newcommand{\bC}{\mathbb{C}}
\newcommand{\bN}{\mathbb{N}}
\newcommand{\bQ}{\mathbb{Q}}
\newcommand{\bP}{\mathbb{P}}

\newcommand{\bR}{\mathbb{R}}
\newcommand{\V}{\mathbb{V}}
\newcommand{\bZ}{\mathbb{Z}}

\newcommand{\RP}{\mathbb{RP}}

\newcommand{\Ad}{\mathrm{Ad}}
\newcommand{\Hom}{\mathrm{Hom}}

\newcommand{\Map}{\mathrm{Map}}
\newcommand{\Aut}{\mathrm{Aut}}
\newcommand{\End}{\mathrm{End}}
\newcommand{\rank}{\mathrm{rank}}

\newcommand{\diag}{\mathrm{diag}}
\newcommand{\Ker}{\mathrm{Ker}}

\newcommand{\fr}{\mathfrak{r}}

\newcommand{\fsu}{\mathfrak{su}}

\newcommand{\fm}{\mathfrak{m}}

\newcommand{\tSi}{ {\tilde{\Si}} }

\newcommand{\tM}{\tilde{M}}
\newcommand{\tP}{\tilde{P}}

\newcommand{\tg}{\tilde{g}}

\newcommand{\ba}{\bar{a}}
\newcommand{\bb}{\bar{b}}
\newcommand{\bc}{\bar{c}}
\newcommand{\bd}{\bar{d}}

\newcommand{\bV}{\bar{V}}

\newcommand{\Si}{ {\Sigma} }
\newcommand{\ep}{\epsilon}

\newcommand{\ab}{a_1,b_1,\ldots,a_\ell,b_\ell}

\newcommand{\pab}{\prod_{i=1}^\ell[a_i,b_i]}

\newcommand{\flU}[2]{ X_{\mathrm{flat}}^{ {#1}, {#2}}(U(n)) }
\newcommand{\ymU}[2]{ X_{\mathrm{YM} }^{{#1},{#2}}(U(n)) }

\newcommand{\ZymU}[1]{ Z_{\mathrm{YM} }^{\ell, {#1}}(U(n)) }

\newcommand{\ymS}[2]{ X_{\mathrm{YM} }^{{#1},{#2}}(U(1)) }

\newcommand{\ZymS}[1]{ Z_{\mathrm{YM} }^{\ell, {#1}}(U(1)) }
\newcommand{\kn}{{\frac{k}{n},\ldots,\frac{k}{n} }}

\newtheorem{thm}{Theorem}

\newtheorem{lm}[thm]{Lemma}
\newtheorem{rem}[thm]{Remark}
\newtheorem{cor}[thm]{Corollary}

\newtheorem{df}[thm]{Definition}
\newtheorem{con}[thm]{Conjecture}
\newtheorem{nota}[thm]{Notation}

\begin{document}

\parskip=0.35\baselineskip

\baselineskip=1.2\baselineskip

\title{Antiperfect Morse Stratification}

\author{Nan-Kuo Ho}
\address{Department of Mathematics, National Tsing Hua University, Hsinchu 300,  Taiwan}
\email{nankuo@math.nthu.edu.tw}

\author{Chiu-Chu Melissa Liu}
\address{Department of Mathematics,
Columbia University, New York, NY 10027, USA}
\email{ccliu@math.columbia.edu}

\date{September 1, 2009}

% \keywords{}

% \subjclass{53}

\begin{abstract}
For an equivariant Morse stratification which contains a unique open stratum,
we introduce the notion of equivariant antiperfection,
which means the difference of the equivariant
Morse series and the equivariant Poincar\'{e} series achieves the
maximal possible value (instead of the minimal possible value $0$ in the
equivariantly perfect case).
We also introduce a weaker condition of local equivariant antiperfection.
We prove that the Morse stratification of the Yang-Mills functional
on the space of connections on a principal $G$-bundle
over a connected, closed, nonorientable surface
$\Si$ is locally equivariantly $\bQ$-antiperfect
when $G=U(2), SU(2), U(3), SU(3)$; we propose
that the Morse stratification is actually equivariantly $\bQ$-antiperfect
in these cases. Our proposal yields formulas of
Poincar\'{e} series $P_t^G(\Hom(\pi_1(\Si),G);\bQ)$
when $G=U(2), SU(2), U(3), SU(3)$.
Our $U(2)$, $SU(2)$ formulas agree with formulas proved by
T. Baird, who also verified our conjectural $U(3)$ formula.
\end{abstract}

\maketitle

\tableofcontents
%---------------

\section{Introduction}

Let $f$ be a Morse function on a compact manifold $M$, so
that it has finitely many isolated nondegenerate
critical points. The Morse polynomial of $f$
is defined to be
$$
M_t(f)=\sum_{p\in \mathrm{Crit}(f)} t^{\lambda_p}
$$
where $\mathrm{Crit}(f)$ is the set of critical points of $f$, and
$\lambda_p$ is the Morse index of $p$. The Morse polynomial of any
Morse function satisfies the Morse inequalities
$$
M_t(f) - P_t(M;K) = (1+t)R_t(K)
$$
where $P_t(M;K)$ is the Poincar\'{e} polynomial of $M$
relative to a coefficient field $K$,
and $R_t(K)$ is a polynomial with nonnegative
coefficients. A Morse function is called
$K$-perfect if $R_t(K)=0$.

In \cite{ym}, Atiyah and Bott  studied Morse theory in a much more general setting:
the manifold $M$ is an infinite dimensional
affine space $\cA$ of connections
on a principal $G$-bundle $P$ over a Riemann surface $\Si$,
where $G$ is a compact connected Lie group;
the functional $f$ is the Yang-Mills functional
$L:\cA\to \bR$, $A\mapsto  \| F_A\|_{L^2}$, which is Morse-Bott instead of
Morse\footnote{Indeed, $L$ is not Morse-Bott in the strict sense,
since its critical sets $\cN_\mu$ are singular in general,
but the Morse index $\lambda_\mu$ of $\cN_\mu$ is well-defined, and
$$
M_t^\cG(L;K)=\sum_{\lambda\in \Lambda} t^{\lambda_\mu}
P_t^{\cG}(\cN_\mu;K)= \sum_{\lambda\in \Lambda} t^{\lambda_\mu}
P_t^{\cG}(\cA_\mu;K)
$$
where $\cA_\mu$ is the stable  manifold of $\cN_\mu$.};
the Yang-Mills functional is invariant under the action
of the gauge group $\cG=\Aut(P)$, and Atiyah and Bott
consider the $\cG$-equivariant
Morse series $M_t^\cG(L;K)$ of $L$ and
$\cG$-equivariant Poincar\'{e} series
$P_t^\cG(\cA;K)$ of $\cA$
(they are infinite series instead of
polynomials). The Morse inequalities in
this context are encoded in the equation
\begin{equation}\label{eqn:inequalities}
M_t^{\cG}(L;K) - P_t^{\cG}(\cA;K) = (1+t) R_t^\cG(K)
\end{equation}
where $R_t^\cG(K)$ is a formal power series with nonnegative coefficients.
When $K=\bQ$, Atiyah and Bott computed
$P_t^\cG(\cA;\bQ)=P_t(B\cG;\bQ)$, where $B\cG$ is the classifying
space of the gauge group\footnote{Atiyah-Bott computed
$P_t(B\cG;\bQ)$ for $G=U(n)$ in \cite[Section 2]{ym}; their
computation can be generalized to any compact connected Lie group
$G$ \cite[Theorem 3.3]{LR}.}, and proved that the Morse
stratification of the Yang-Mills functional is equivariantly
$\bQ$-perfect, in the sense that $R_t^\cG(\bQ)=0$ (when $G=U(n)$,
they also proved that it is equivariantly $\bZ_p$-perfect for any
prime $p$). This leads to a recursive formula computing the
equivariant Poincar\'{e} series $P_t^\cG(\cA_{ss};\bQ)$ of the
unique open stratum $\cA_{ss}\subset \cA$. When the obstruction
class $o_2(P)\in H^2(\Si;\pi_1(G))\cong \pi_1(G)$ is torsion, the
absolute minimum of the Yang-Mills functional is zero, and the
unique open stratum $\cA_{ss}$ is the stable manifold of the space
$\cN_0$ of flat connections on $P$. We have
\begin{equation}
P_t^\cG(\cA_{ss};\bQ)= P_t^\cG(\cN_0;\bQ)
= P_t^G\left(\Hom(\pi_1(\Si),G)_P ;\bQ\right)
\end{equation}
where the subscript $P$ labels the connected
component corresponding to the topological
type $P$ (which is classified by
the obstruction class $o_2(P)$).

In \cite{HL1, HL2}, the authors generalized
some aspects of \cite{ym} to
connected, closed, {\em nonorientable} surfaces.
Let $\Si$ be a connected, closed nonorientable
surface, so that it is the connected
sum of $m>0$ copies of $\RP^2$. Let
$\pi:\tSi\to \Si$ be the orientable double cover,
so that $\tSi$ is a Riemann surface of genus $m-1$.
Let $\cA$ and $\tilde{\cA}$ denote
the spaces of connections on a principal
$G$-bundle $P\to\Si$ and on the pull back
$\pi^*P\to \tSi$, respectively. Then $A\to \pi^*A$
defines an inclusion $\cA\hookrightarrow \tilde{\cA}$
whose image is the fixed locus of an anti-holomorphic,
anti-symplectic involution $\tau$ on $\tilde{\cA}$, and the Yang-Mills functional
$L:\cA\to \bR$ is, by definition, the restriction of
the Yang-Mills functional on $\tilde{\cA}$ to $\cA$.
The absolute minimum of the Yang-Mills functional $L:\cA\to \bR$
is always zero, achieved by flat connections. The normal bundles of Morse strata of
$\cA$ defined by $L$ are
real vector bundles,  so a priori one can
only take $K=\bZ_2$. Together
with  Ramras, the authors proved that these bundles,
and their associated homotopy orbit bundles, are
orientable when $G=U(n)$ or $SU(n)$ \cite{HLR},
so we may use any field coefficient in this case.
When $G=U(n)$ or $SU(n)$, the Morse stratification of $\cA$
defined by $L$ is
not equivariantly $\bQ$-perfect.

In this paper, we introduce the notion of equivariant
$K$-antiperfection,  which means the discrepancy $R_t^\cG(K)$
in \eqref{eqn:inequalities} achieves
the maximal possible value (instead of
the minimal possible value $0$ in the
perfect case).
We also introduce a weaker condition of local
equivariant $K$-antiperfection. We prove
that the Morse stratification
defined by the Yang-Mills functional on
the space of connections on a principal
$G$-bundle over a connected, closed,
nonorientable surface $\Si$
is locally equivariantly $\bQ$-antiperfect
when $G=U(2), SU(2), U(3), SU(3)$; we propose
that it is actually equivariantly $\bQ$-antiperfect
in these cases. (When $G=U(1)$, there is
only one stratum $\cA_{ss}=\cA$.) Our proposal
yields formulas for the following equivariant Poincar\'{e} series when $n=2,3$:
\begin{eqnarray*}
&& P_t^{U(n)}(\Hom(\pi_1(\Si), U(n))_+;\bQ),\quad
P_t^{U(n)}(\Hom(\pi_1(\Si), U(n))_-;\bQ),\\
&& P_t^{SU(n)}(\Hom(\pi_1(\Si), SU(n));\bQ),
\end{eqnarray*}
where  $+$ and $-$ label the components
corresponding to the trivial and nontrivial
$U(n)$-bundles over $\Si$, respectively.
Indeed we show that these formulas
hold if and only if equivariant $\bQ$-antiperfection holds
in the rank 2 and rank 3 cases.
Our rank 2 formulas \eqref{eqn:Utwo-plus}, \eqref{eqn:Utwo-minus},
\eqref{eqn:SUtwo} agree with formulas proved by
T. Baird \cite{B1}. During the revision of this paper, 
Baird established equivariant $\bQ$-antiperfection in the $U(3)$ case,
and thus verified our conjectural $U(3)$ formula \eqref{eqn:Uthree} 
\cite{B3}. 

\subsection*{Acknowledgments}
We thank Daniel Ramras, Michael Thaddeus, and  Graeme Wilkin
for helpful communications. We thank
Thomas Baird for pointing out a gap in
the first version of this paper.
The first author was partially supported by an NSC grant 
97-2628-M-006-013-MY2.
The second author was partially supported by
the Sloan Research Fellowship.

\section{Preliminaries, Definitions, and Statements of Results}

\subsection{Morse stratification}
Let $\cA$ be the space of connections
on a principal $U(n)$-bundle or $SU(n)$-bundle $P$ over a connected, closed, orientable
or nonorientable surface $\Si$, and let $\cG=\Aut(P)$ be the
group of unitary gauge transformations.  $\cA$ is an infinite
dimensional affine space, equipped with a $\cG$-invariant
Riemannian metric.   The Yang-Mills functional
$L:\cA\to \bR$ is invariant under the action of
the gauge group $\cG$, and defines a $\cG$-equivariant
Morse stratification
\begin{equation}\label{eqn:strata}
\cA =\bigcup_{\mu\in \Lambda} \cA_\mu =\cA_{ss}\cup \bigcup_{\mu\in \Lambda'} \cA_\mu
\end{equation}
where $\cA_{ss}$ is the unique open stratum.
When $n=1$, there is only one stratum: $\cA=\cA_{ss}$. From
now on we will assume $n>1$.

The index set $\Lambda$ is partially ordered
such that given $I\subset \Lambda$,
$\cA_I:= \displaystyle{\bigcup_{\lambda \in I}\cA_\lambda}$
is open if $\lambda\in I$, $\mu\leq \lambda \Rightarrow  \mu\in I$;
this partial ordering can be refined to a total ordering \cite[Section 2]{R}.
In the following discussion, {\em we fix a total ordering on $\Lambda$} so that we have
a filtration of $\cA$ by open subsets.
Given $\mu\in \Lambda'$, let $J= \{\lambda\in \Lambda\mid \lambda \leq \mu\}$, and let
$I= J - \{\mu\}$, so that  $\cA_I\subset \cA_J\subset \cA$
are inclusions of open subsets. We have the following isomorphisms of
$\cG$-equivariant cohomology groups:
\begin{equation}\label{eqn:excision}
H^k_\cG (\cA_J,\cA_I)\stackrel{\textup{excision} }{\cong}
H^k_\cG( (\cA_\mu)_\epsilon, (\cA_\mu)_\epsilon-\cA_\mu)
\stackrel{\textup{Thom isomorphism} }{\cong} H^{k-\lambda_\mu}_\cG (\cA_\mu)
\end{equation}
where $(\cA_\mu)_\ep$ is a $\cG$-equivariant tubular neighborhood of
$\cA_\mu$ in $\cA_J$ (see \cite[Section 3]{R} for a construction of
$(\cA_\mu)_\ep$), and $\lambda_\mu$ is the rank of the
normal bundle $\bN_\mu$ of $\cA_\mu$ in $\cA$. The normal bundle
$\bN_\mu\to \cA_\mu$ is  a $\cG$-equivariant complex vector bundle when $\Si$ is orientable, and
is a $\cG$-equivariant orientable real vector bundle
when $\Si$ is nonorientable \cite{HLR}
(when $\Si$ is the Klein bottle, we assume that $n=2$ or $3$),
so the Thom isomorphism in \eqref{eqn:excision} holds for any coefficient ring.
We may identify the pair $\left((\cA_\mu)_\ep, (\cA_\mu)_\ep-\cA_\mu\right)$ with
$\left(\bN_\mu, (\bN_\mu)_0\right)$, where $(\bN_\mu)_0$ is the
complement of the zero section of the vector bundle
$\bN_\mu\to \cA_\mu$.  We have the following
commutative diagram for any coefficient ring:

\begin{equation}\label{eqn:diagram}
\begin{CD}
H^k_\cG(\cA_J,\cA_I) @>{\alpha^k}>> H^k_\cG(\cA_J) @>{\beta^k}>>
H^k_\cG(\cA_I)
@>{\gamma^k}>> H^{k+1}_\cG(\cA_J,\cA_I) @>{\alpha^{k+1}}>>\cdots \\
@VV{\cong}V @VV{j^k}V @VV{i^k}V @VV\cong V\\
H^k_\cG(\bN_\mu,(\bN_\mu)_0) @>{\alpha^k_\ep}>> H^k_\cG(\bN_\mu)
@>{\beta^k_\epsilon}>> H^k_\cG((\bN_\mu)_0) @>{\gamma^k_\epsilon}>>
 H^{k+1}_\cG(\bN_\mu,(\bN_\mu)_0) @>{\alpha^{k+1}_\ep}>> \cdots \\
@VV{\cong}V  @VV{ \stackrel{s^k}{\cong} }V \\
H^{k-\lambda_\mu}_\cG(\cA_\mu) @>{\cup e_\cG(\bN_\mu)}>> H^k_\cG(\cA_\mu)
\end{CD}
\end{equation}
where $i^k$, $j^k$, $s^k$ are induced by inclusions. From the above
Diagram \eqref{eqn:diagram} we see that $\Ker(\alpha^k)\subset \Ker(\alpha^k_\ep)$
under the identification $H^k_\cG(\cA_J,\cA_I)\cong
H^k_\cG(\bN_\mu,(\bN_\mu)_0)$.

\subsection{Morse inequalities}\label{sec:inequalities}
We now consider field coefficient $K$, so that
the cohomology groups are vector spaces over $K$.
For any $\mu\in \Lambda'$, we define
$$
Z^k_\cG(\cA_\mu;K)= \Ker\left(H^k_\cG (\cA_\mu;K)\cong
H^{k+\lambda_\mu}_\cG(\cA_J,\cA_I;K)
\stackrel{\alpha^{k+\lambda_\mu}}{\longrightarrow}
H^{k+\lambda_\mu}_\cG(\cA_J;K)\right)
$$
so that $Z^k_\cG(\cA_\mu;K)$ is a subspace of $H^k_{\cG}(\cA_\mu;K)$.
We have an exact sequence
\begin{equation}\label{eqn:five}
0 \to H^{k-\lambda_\mu}_\cG (\cA_\mu;K)/Z^{k-\lambda_\mu}_\cG (\cA_\mu;K)
\to H^k_\cG(\cA_J;K)\to H^k_\cG(\cA_I;K)
\to Z_\cG^{k+1-\lambda_\mu}(\cA_\mu;K)\to 0.
\end{equation}
Define a power series
$$
Z_t^\cG(\cA_\mu;K)=\sum_{k=0}^\infty t^k \dim_K Z^k_\cG(\cA_\mu;K) \in \bZ[[t ]].
$$
Then the exact sequence \eqref{eqn:five} implies
\begin{equation}\label{eqn:JI}
 P_t^\cG(\cA_J;K) + (1+t) t^{\lambda_\mu-1} Z_t^\cG(\cA_\mu;K)
= P_t^\cG(\cA_I;K) + t^{\lambda_\mu} P_t^\cG(\cA_\mu;K).
\end{equation}

Given two power series $p(t), q(t)\in \bZ[[t ]]$, we say
$p(t)\leq q(t)$ if $q(t)-p(t)$ is a power series with
nonnegative coefficients.  Then
\begin{equation}\label{eqn:ZPmu}
0\leq Z_t^\cG(\cA_\mu;K) \leq P_t^\cG(\cA_\mu;K).
\end{equation}

Define
$$
R_t^{\cG}(K) = \sum_{\mu\in \Lambda'} t^{\lambda_\mu-1} Z^\cG_t(\cA_\mu;K),\quad
\tilde{M}_t^\cG(K) =\sum_{\mu\in \Lambda'} t^{\lambda_\mu-1} P_t^\cG(\cA_\mu;K).
$$
Note that $\lambda_\mu-1\geq 0$ for $\mu\in \Lambda'$, so
$R_t^\cG(K)$, $\tilde{M}_t^\cG(K)$ are power series in $\bZ[[t]]$
with nonnegative coefficients.  The following lemma follows from the
definitions and  \eqref{eqn:ZPmu}.
\begin{lm} \label{lm:RtM}
\begin{equation}\label{eqn:RtM}
0\leq R_t^\cG(K)\leq \tM_t^\cG(K).
\end{equation}
Moreover,
\begin{enumerate}
\item[(i)] $R_t^\cG(K)=0$ if and only if $Z_t^\cG(\cA_\mu;K)=0$ for all $\mu\in \Lambda'$;
\item[(ii)] $R_t^\cG(K)=\tM_t^\cG(K)$ if and only if  $Z_t^\cG(\cA_\mu;K)
=P_t^\cG(\cA_\mu;K)$ for all $\mu\in \Lambda'$.
\end{enumerate}
\end{lm}

\begin{rem}
A priori the definitions $Z_t^\cG(K)$ and $R_t^\cG(K)$ depends on
the choice of the total ordering when such total ordering is not
unique, since the index set $J=\{\lambda\in \Lambda\mid \lambda\leq
\mu\}$ depends on the total ordering. By \eqref{eqn:PRM} below,
$R_t^\cG(K)$ does not depend on the choice. $R_t^\cG(K)$ can be
defined for more general equivariant Morse stratification which
contains a unique open stratum.
\end{rem}

Define the $\cG$-equivariant Morse series of the stratification \eqref{eqn:strata}
as follows.
\begin{df}[Morse series] We define the $\cG$-equivariant Morse series of
the $\cG$-equivariant stratification \eqref{eqn:strata} relative
to the coefficient field $K$ to be

\begin{equation}
M_t^{\cG}(K)= \sum_{\mu\in \Lambda} t^{\lambda_\mu} P_t^\cG(\cA_\mu;K)
=P_t^{\cG}(\cA_{ss};K) + t \tilde{M}_t^\cG(K).
\end{equation}
\end{df}

From  \eqref{eqn:JI} and  \eqref{eqn:RtM} we obtain the following.
\begin{lm}[Morse inequalities]\label{lm:Mneq}
\begin{equation}\label{eqn:PRM}
P_t^\cG(\cA;K) + (1+t) R_t^\cG(K) = M_t^{\cG}(K)
= P_t^\cG(\cA_{ss};K) + t\tilde{M}_t^\cG(K)
\end{equation}
where
$$
0\leq R_t^\cG(K)\leq \tilde{M}_t^\cG(K) =\sum_{\mu\in
\Lambda'}t^{\lambda_\mu-1} P_t(\cA_\mu;K).
$$
Therefore
$$
P_t^{\cG}(\cA;K)-\sum_{\mu\in\Lambda'}t^{\lambda_\mu}P_t^{\cG}(\cA_\mu;K) \leq
P_t^\cG(\cA_{ss};K) \leq
P_t^{\cG}(\cA;K)+ \sum_{\mu\in \Lambda'}t^{\lambda_\mu-1}P_t^{\cG}(\cA_\mu;K).
$$
\end{lm}

\begin{rem}\label{ABT}
When $\Si$ is orientable, Atiyah and Bott proved that
$\alpha^k$ is injective for all $k$ and for all $\mu\in \Lambda'$
when $K=\bQ$ or $K=\bZ_p$ ($p$ any prime) \cite{ym}. So when
$K=\bQ$ or $K=\bZ_p$ ($p$ any prime),  $R_t^\cG(K)=0$, and
$$
P_t^{\cG}(\cA;K)-\sum_{\mu\in\Lambda'}t^{\lambda_\mu}P_t^{\cG}(\cA_\mu;K)
= P_t^\cG(\cA_{ss};K).
$$
\end{rem}

\subsection{Perfect stratification and antiperfect stratification}\label{sec:main}

In Remark \ref{ABT}, the stratification is
said to be equivariantly $K$-perfect.
Motivated by the definition of $K$-perfect
stratification in \cite{ym}
and the extremal cases of the Morse inequalities
(Lemma \ref{lm:Mneq}), we introduce the following
definitions. Let $\alpha^k$ and $\alpha^k_\ep$
be as in Diagram \eqref{eqn:diagram}.

\begin{df}[perfect stratification and antiperfect stratification]\label{global}
We say the $\cG$-equivariant stratification \eqref{eqn:strata} is
{\em equivariantly $K$-perfect} if
$$
\alpha^k: H^k_\cG(\cA_J,\cA_I;K)\to H^k_\cG(\cA_J;K)
$$
is injective for all $k$ and all $\mu\in \Lambda'$;
we say \eqref{eqn:strata} is
{\em equivariantly $K$-antiperfect}
if $\alpha^k=0$ for all $k$ and all $\mu\in \Lambda'$.
\end{df}

\begin{rem}
By Lemma \ref{lm:perfect} and Lemma \ref{lm:anti-perfect} below, the
definitions in Definition \ref{global} do not depend on the choice
of total ordering.
\end{rem}

\begin{df}[locally perfect stratification and locally antiperfect stratification]\label{local}
We say the $\cG$-equivariant stratification \eqref{eqn:strata}
is {\em locally equivariantly $K$-perfect} if
$$
\alpha^k_\ep:H^{k}_\cG(\bN_\mu, (\bN_\mu)_0;K)\to H^k_\cG(\bN_\mu;K)
$$
is injective for all $k$ and all  $\mu\in \Lambda'$;
we say \eqref{eqn:strata}
is {\em locally equivariantly $K$-antiperfect} if
$\alpha^k_\ep=0$ for all $k$ and all $\mu\in \Lambda'$.
\end{df}

\begin{rem}\label{local-global}
Since $\Ker(\alpha_k)\subset \Ker(\alpha_k^\ep)$, it is immediate
from Definition \ref{global} and Definition \ref{local} that
\begin{eqnarray*}
&& \textit{\eqref{eqn:strata} is locally equivariantly $K$-perfect
$\Rightarrow$ \eqref{eqn:strata} is equivariantly $K$-perfect.}\\
&&\textit{\eqref{eqn:strata} is equivariantly $K$-antiperfect
$\Rightarrow$ \eqref{eqn:strata} is locally equivariantly
$K$-antiperfect.}
\end{eqnarray*}
\end{rem}

From Definition \ref{global} and the discussion in Section \ref{sec:inequalities},
we have the following equivalent conditions of
equivariant perfection and antiperfection.
\begin{lm}[reformulation of equivariant perfection]\label{lm:perfect}
The following conditions are equivalent:
\begin{enumerate}
\item[P1.] \eqref{eqn:strata} is an equivariantly $K$-perfect stratification.
\item[P2.] For any $\mu\in \Lambda'$, the long exact sequence
\begin{equation}\label{eqn:LES}
\cdots\to H^{k-\lambda_\mu}_{\cG}(\cA_\mu;K)\to H^k_{\cG}(\cA_J;K)\to H^k_\cG(\cA_I;K)\to
H^{k+1-\lambda_\mu}_{\cG}(\cA_\mu;K)\to \cdots
\end{equation}
breaks into short exact sequences
$$
0\to H^{k-\lambda_\mu}_{\cG}(\cA_\mu;K)\to H^k_{\cG}(\cA_J;K)\to H^k_{\cG}(\cA_I;K) \to 0.
$$

\item[P3.] $Z_t^\cG(\cA_\mu;K)=0$ for all $\mu\in \Lambda'$.
\item[P4.] $R_t^\cG(K)=0$.
\item[P5.] $\displaystyle{
P_t^\cG(\cA;K)=P_t^\cG(\cA_{ss};K)+\sum_{\mu\in \Lambda'}t^{\lambda_\mu} P_t^\cG(\cA_\mu;K)}$.
\end{enumerate}
\end{lm}

\begin{lm}[reformulation of equivariant antiperfection]\label{lm:anti-perfect}
The following conditions are equivalent:
\begin{enumerate}
\item[A1.] \eqref{eqn:strata} is an equivariantly $K$-antiperfect stratification.
\item[A2.] For any $\mu\in \Lambda'$, the long exact sequence \eqref{eqn:LES}
breaks into short exact sequences
$$
0\to H^k_{\cG}(\cA_J;K)\to H^k_{\cG}(\cA_I;K) \to H^{k+1-\lambda_\mu}_\cG(\cA_\mu;K)\to 0.
$$
\item[A3.] $Z_t^\cG(\cA_\mu;K)=P_t^\cG(\cA_\mu;K)$ for all $\mu\in \Lambda'$.
\item[A4.] $R_t^\cG(K)=\tilde{M}_t^\cG(K)$.
\item[A5.] $\displaystyle{
P_t^\cG(\cA_{ss};K)=P_t^\cG(\cA;K)+\sum_{\mu\in \Lambda'}t^{\lambda_\mu-1} P_t^\cG(\cA_\mu;K)}$.
\end{enumerate}
\end{lm}

%----------------
When $\Si$ is orientable, and $K=\bQ$ or $K=\bZ_p$ ($p$ any prime),
Atiyah and Bott showed that $e_\cG(\bN_\mu)$ in the commutative Diagram
\eqref{eqn:diagram} is not a zero divisor in $H^*_{\cG}(\cA_\mu;K)$, so $\alpha^k_\ep$
is injective. We may reformulate this result as follows.
\begin{thm}[Atiyah-Bott] \label{thm:local-orientable}
Let $\cA$ be the space of connections
on a principal $U(n)$-bundle or $SU(n)$-bundle over a Riemann surface.
Let $K=\bQ$ or $K=\bZ_p$ ($p$ any prime).
Then  the stratification \eqref{eqn:strata} is
locally equivariantly $K$-perfect; therefore
it is equivariantly $K$-perfect.
\end{thm}

\subsection{Yang-Mills theory on a closed nonorientable surface}
Let $\Si$ be a connected, closed, nonorientable surface.
Let $\tP$ be the pull back of $P$ to the orientable double
cover $\tSi\to \Si$. Recall that a stratum $\cA_\mu\subset \cA$ corresponds to
reduction of the structure group of $\tP\to \tSi$ (instead of
$P\to \Si$) to a subgroup
$$
U(n_1)\times \cdots \times U(n_r) \subset U(n)
$$
or
$$
\left( U(n_1)\times \cdots \times U(n_r)\right) \cap SU(n)\subset SU(n)
$$
where $n_1+\cdots+n_r=n$.
We say $\mu$ contains a rank 1 factor if
$n_j=1$ for some $j$.  In particular, when $n=2$ or $3$,
every $\mu\in \Lambda'$ contains a rank 1 factor
(see Section \ref{sec:euler}).

In Section \ref{sec:euler}, we prove the following:
\begin{thm}[vanishing of equivariant Euler class]\label{thm:vanishing}
Let $\cA$ be the space of connections on a principal
$U(n)$-bundle or $SU(n)$-bundle ($n>1$) over a connected, closed, nonorientable surface $\Si$.
When $\chi(\Si)=0$, so that $\Si$ is homeomorphic to
the Klein bottle, we assume in addition that $n\leq 3$.
We use rational coefficient $\bQ$.
\begin{enumerate}
\item[(i)] If $\Si=\RP^2$ then $e_{\cG}(\bN_\mu)=0$
for all $\mu\in \Lambda'$
\item[(ii)] If $\Si$ is not homeomorphic to $\RP^2$ then
$e_{\cG}(\bN_\mu)=0$ if $\mu$ contains a rank 1 factor.
\end{enumerate}
Therefore $\alpha^k_\ep=0$ for all $k$ in the above
two cases.
\end{thm}

\begin{cor}\label{thm:local-anti-perfect}
Let $\cA$ be the space of connections on a principal $U(n)$-bundle
or $SU(n)$-bundle ($n>1$) over a connected, closed, nonorientable surface $\Si$. Then
the Morse stratification \eqref{eqn:strata} is locally equivariant
$\bQ$-antiperfect in the following cases:
\begin{enumerate}
\item[(i)] $\Si=\RP^2$, $n$ any positive integer greater than 1;
\item[(ii)] $\Si$ is not homeomorphic to $\RP^2$, $n=2$ or $3$.
\end{enumerate}
\end{cor}

Although local equivariant antiperfection does not imply
equivariant antiperfection, it is natural to ask if
equivariant $\bQ$-antiperfection holds in the cases listed in
Corollary \ref{thm:local-anti-perfect}.
%---------------------
\begin{nota}\label{notation}
Given a principal bundle $P$ over a connected,
closed, orientable or nonorientable surface, let
$\cA(P)$ denote the space of connections
on $P$, and let $\cN_0(P)$ denote
the space of flat connections on $P$.
Let $\cG(P)=\Aut(P)$ and $\cG_0(P)$
be the gauge group and the based gauge group,
respectively.
\end{nota}

Let $\Si$ be a closed, connected, nonorientable surface,
so that it is the connected sum of $m>0$ copies of
$\RP^2$. Then the topological type of a principal
$U(n)$-bundle $P\to \Si$ is classified by
$c_1(P) \in H^2(\Si;\bZ)\cong \bZ/2\bZ$.
Let $P^{n,+}_\Si$ and $P^{n,-}_\Si$ denote
the trivial ($c_1=0$ mod 2) and nontrivial
($c_1=1$ mod 2) principal $U(n)$-bundle over $\Si$,
and let $Q^n_\Si$ be a principal $SU(n)$-bundle
over $\Si$ (which must be topologically trivial).

We have
$$
\Hom(\pi_1(\Si), U(n))=\Hom(\pi_1(\Si), U(n))_{+1}
\cup \Hom(\pi_1(\Si),U(n))_{-1},
$$
where
$$
\Hom(\pi_1(\Si), U(n))_{\pm 1}   \cong \cN_0(P_\Si^{n,\pm})/\cG_0(P_\Si^{n,\pm}).
$$
We also have
$$
\Hom(\pi_1(\Si), SU(n)) \cong \cN_0(Q_\Si^n)/\cG_0(Q_\Si^n).
$$
When $m>1$, $\Hom(\pi_1(\Si), U(n))_{+1}$ and $\Hom(\pi_1(\Si), U(n))_{-1}$
are the two connected components
of $\Hom(\pi_1(\Si),U(n))$, and $\Hom(\pi_1(\Si), SU(n))$ is connected.
When $m=1$,
$$
\Hom(\pi_1(\RP^2), U(n))_{\pm 1}=\{ a\in U(n)\mid a^2 =I_n, \det(a)=\pm 1\}.
$$
$\Hom(\pi_1(\RP^2), U(n))_{+1}= \Hom(\pi_1(\RP^2), SU(n))$ is disconnected for $n\geq 2$,  and
$\Hom(\pi_1(\RP^2),U(n))_{-1}$ is disconnected for $n\geq 3$.

We derive the following result in Section \ref{sec:Utwo}.

\begin{thm}[equivariant Poincar\'{e} series, rank 2 case]\label{thm:Utwo}
Let $\Si$ be a connected, closed, nonorientable surface, and
let $\tg$ be the genus of the
orientable double cover $\tSi$ of $\Si$. Then
the stratifications \eqref{eqn:strata} on $\cA(P^{2,+}_\Si)$,
$\cA(P^{2,-}_\Si)$, and $\cA(Q^2_\Si)$ are equivariantly $\bQ$-antiperfect if and only if
the following  \eqref{eqn:Utwo-plus}, \eqref{eqn:Utwo-minus}, and
\eqref{eqn:SUtwo} hold, respectively:
\begin{equation}\label{eqn:Utwo-plus}
P_t^{U(2)}\left(\Hom(\pi_1(\Si), U(2))_{(-1)^{\tg} };\bQ\right)
=\frac{(1+t)^{\tg} }{(1-t^2)(1-t^4)}
( (1+t^3)^{\tg} + t^{\tg}(1+t)^{\tg} )
\end{equation}
\begin{equation}\label{eqn:Utwo-minus}
P_t^{U(2)}\left(\Hom(\pi_1(\Si), U(2))_{(-1)^{\tg+1}};\bQ\right)
=\frac{(1+t)^{\tg}}{(1-t^2)(1-t^4)} ( (1+t^3)^{\tg} +
t^{\tg+2}(1+t)^{\tg} )
\end{equation}
\begin{equation}\label{eqn:SUtwo}
P_t^{SU(2)}\left(\Hom(\pi_1(\Si), SU(2));\bQ\right)
=\begin{cases}
\displaystyle{\frac{(1+t^3)^{\tg} + t^{\tg}(1+t)^{\tg} }{1-t^4 } }, & \tg \textup{ is even},\\
\displaystyle{\frac{(1+t^3)^{\tg} + t^{\tg+2}(1+t)^{\tg}}{1-t^4} }, & \tg \textup{ is odd}.
\end{cases}
\end{equation}
\end{thm}

The formulas in Theorem \ref{thm:Utwo} have been proved by T. Baird \cite{B1}:
\begin{thm}[{Baird}]\label{thm:baird-two}
\eqref{eqn:Utwo-plus}, \eqref{eqn:Utwo-minus}, and \eqref{eqn:SUtwo}
hold for any $\tg \geq 0$.
\end{thm}

From Theorem \ref{thm:Utwo} and Theorem \ref{thm:baird-two}, we conclude that
\begin{cor}
Let $\cA$ be the space of connections
on a principal $U(2)$-bundle or $SU(2)$-bundle over a connected, closed, nonorientable surface.
Then the Morse stratification \eqref{eqn:strata} on $\cA$ is equivariantly
$\bQ$-antiperfect.
\end{cor}

We derive the following result in Section \ref{sec:Uthree}.

\begin{thm}[equivariant Poincar\'{e} series, rank 3 case]\label{thm:Uthree}
Let $\Si$ be a connected, closed, nonorientable surface, and let $\tg$ be the genus of the
orientable double cover $\tSi$ of $\Si$. Then the stratifications
\eqref{eqn:strata} on $\cA(P^{3,\pm}_\Si)$ and
$\cA(Q^3_\Si)$ are equivariantly $\bQ$-antiperfect if and only if
the following \eqref{eqn:Uthree} and \eqref{eqn:SUthree} hold, respectively.
\begin{equation}\label{eqn:Uthree}
\begin{aligned}
& P_t^{U(3)}\left(\Hom(\pi_1(\Si), U(3))_{+1};\bQ\right)
=P_t^{U(3)}\left(\Hom(\pi_1(\Si), U(3))_{-1};\bQ\right)\\
=& \frac{(1+t)^{\tg} } {(1-t^2)(1-t^4)(1-t^6)} \left(
(1+t^3)^{\tg}(1+t^5)^{\tg} + t^{3\tg}(1+t)^{2\tg} (1+t^2+t^4)\right)
\end{aligned}
\end{equation}
\begin{equation}\label{eqn:SUthree}
P_t^{SU(3)}\left(\Hom(\pi_1(\Si), SU(3) );\bQ\right)=
\frac{(1+t^3)^{\tg}(1+t^5)^{\tg} + t^{3\tg}(1+t)^{2\tg} (1+t^2+t^4)
} {(1-t^4)(1-t^6)}
\end{equation}
\end{thm}

Motivated by Theorem \ref{thm:vanishing} and
Theorem \ref{thm:Uthree}, we make the following
conjecture.
\begin{con}\label{conjecture}
\eqref{eqn:Uthree} and \eqref{eqn:SUthree} hold for any $\tg\geq 0$.
\end{con}

During the revision of this paper, T. Baird showed
that the stratifications \eqref{eqn:strata} on $\cA(P^{n,+}_\Si)$
and $\cA(P^{n,-}_\Si)$ are
equivariantly $\bQ$-antiperfect
for $n=3$, but not equivariantly $\bQ$-antiperfect
for $n\geq 4$ \cite{B3}. Therefore
\eqref{eqn:Uthree} holds for any $\tg\geq 0$.
 Using a different approach, Baird verified our conjectural
 rank 3 formulas \eqref{eqn:Uthree}  and \eqref{eqn:SUthree} when 
$\Si$ is the real projective plane ($\tg=0$) or the Klein bottle ($\tg=1$) \cite{B2}.

\section{Equivariant Euler Class}\label{sec:euler}

Let $\Si$ be a connected, closed, nonorientable surface, so that it
is the connected sum  of $m>0$ copies of $\RP^2$, and let
$\pi:\tSi\to \Si$ be the orientable double cover, so that $\tSi$ is
a Riemann surface of genus $\tg=m-1$. Let $P^{n,+}_\Si$ and
$P^{n,-}_\Si$ be defined as in Section \ref{sec:main}, and let
$P^{n,k}_\tSi$ denote the degree $k$ principal $U(n)$-bundle on
$\tSi$. Then $\pi^*P^{n,\pm }_\Si\cong P^{n,0}_\tSi\cong U(n)\times
\tSi$ is a trivial $U(n)$-bundle over $\tSi$.

Let $\cA(P)$, $\cN_0(P)$, $\cG(P)$, and $\cG_0(P)$ be
defined as in Notation \ref{notation}.
There is an inclusion $\cA(P^{n,\pm}_\Si)\hookrightarrow
\cA(P^{n,0}_\tSi)$ defined by $A\mapsto \pi^*A$, and
the image is the fixed locus of
an anti-symplectic, anti-holomorphic
involution $\tau^\pm$ on $\cA(P^{n,0}_\tSi)$.
The Yang-Mills functional on $\cA(P^{n,\pm}_\Si)$ is, by
definition, restriction of the Yang-Mills functional on
$\cA(P^{n,0}_\tSi)$ to
$\cA(P^{n,0}_\tSi)^{\tau^\pm}$. The Yang-Mills functional on
$\cA(P^{n,0}_\tSi)$ and the metric on $\cA(P^{n,0}_\tSi)$ are
invariant under the involutions $\tau^+$, $\tau^-$. The
Morse strata of $\cA(P^{n,\pm}_\Si)$ are of the form $\cA_\mu
=\tilde{\cA}_\mu \cap \cA(P^{n,0}_\tSi)^{\tau^\pm}$, where
$\tilde{\cA}_\mu$ is a Morse stratum of $\cA(P^{n,0}_\tSi)$.
The Yang-Mills functional is invariant under the action of the gauge group, and
each Morse stratum is preserved by the action of the gauge group.
Since the arguments for $\cA(P^{n,-}_\Si)$ and $\cA(P^{n,+}_\Si)$
are the same, we will use the notation $\cA$ instead of
$\cA(P^{n,\pm}_\Si)$ when there is no confusion.

\subsection{Atiyah-Bott types}\label{sec:ABtypes}
The Morse strata on $\cA(P^{n,k}_\tSi)$ are labeled by the
Atiyah-Bott types $\mu\in I_{n,k}$, where
\begin{eqnarray*}
I_{n,k}&=&\Bigl \{ \mu= (\mu_1,\ldots,\mu_n)=  \Bigl(
\underbrace{\frac{k_1}{n_1},\ldots,\frac{k_1}{n_1}}_{n_1},\ldots,
\underbrace{\frac{k_m}{n_m},\ldots,\frac{k_m}{n_m}}_{n_m}\Bigr )\Bigr | \\
&& \Bigl. n_j\in \bZ_{>0}, k_j\in \bZ, \sum_{j=1}^m n_j= n,
\sum_{j=1}^m k_j=k, \frac{k_1}{n_1} >\cdots >\frac{k_m}{n_m} \Bigr
\}
\end{eqnarray*}
The Morse stratification on $\cA(P^{n,k}_\tSi)$ is given by
$$
\cA(P^{n,k}_\tSi)=\bigcup_{\mu\in I_{n,k}} \tilde{\cA}_\mu.
$$
The unique open stratum is
$$
\cA(P^{n,k}_\tSi)_{ss}=\tilde{\cA}_{\kn}.
$$
The partial ordering on $I_{n,k}$ is given by
$$
\mu\geq\nu \quad
\mbox{iff} \quad \sum_{j\leq i}\mu_j\geq\sum_{j\leq i}\nu_j,
\quad\forall~ i=1,\ldots,n-1.$$

The involution $\tau^\pm$ acts on strata by
$\cA_\mu\mapsto \cA_{\tau_0(\mu)}$, where
$$
\tau_0:I_{n,0}\to I_{n,0},\quad
(\mu_1,\ldots,\mu_n)\mapsto (-\mu_n,\ldots,-\mu_1).
$$
Using the same notation as in
\cite[Section 7.1]{HL1}, denote $I_n=I_{n,0}^{\tau_0}$ the fixed
point set of $\tau_0$ on $I_{n,0}$. Then any $\mu\in I_n$ is of
the form
\begin{equation}\label{eqn:mu}
\mu=\Bigl( \underbrace{\frac{k_1}{n_1},\ldots,
\frac{k_1}{n_1}}_{n_1},\ldots,
 \underbrace{\frac{k_r}{n_r},\ldots, \frac{k_r}{n_r}}_{n_r}
,\underbrace{0,\ldots,0}_{n_0},
 \underbrace{-\frac{k_r}{n_r},\ldots, -\frac{k_r}{n_r}}_{n_r},
\ldots,
 \underbrace{-\frac{k_1}{n_1},\ldots, -\frac{k_1}{n_1}}_{n_1}\Bigr)
\end{equation}
where
$$
\frac{k_1}{n_1}>\cdots > \frac{k_r}{n_r} >0, \quad n_0\geq 0,\quad
n_i>0, \quad 2(n_1+\cdots+n_r)+n_0=n,
$$
Define
$$
I_n^0=\{ \mu\in I_n\mid \mu_i=0 \textup{ for some }i \}.
$$
For $\mu\in I_n^0$, $\tilde{\cA}_\mu$
intersects both $\cA(P^{n,0}_\tSi)^{\tau^+}$
and $\cA(P^{n,0}_\tSi)^{\tau^-}$. Note that
$I_n=I_n^0$ when $n$ is odd.

When $n=2n'$ is even,
any $\mu\in I_n\setminus I_n^0$ is of the form
\begin{equation}\label{eqn:I-pm}
\mu=(\nu,\tau_0(\nu)), \quad \nu\in I_{n',k},\quad \nu_1>\ldots>\nu_{n'}>0.
\end{equation}
By \cite[Section 7.1]{HL1}, $\tilde{\cA}_\mu$ intersect
$\cA(P^{n,0}_\tSi)^{\tau^+}$ (resp. $\cA(P^{n,0}_\tSi)^{\tau^-}$) if
and only if $n'\chi(\Si) + k$ is even (resp. odd). Here $\chi(\Si)$ is the
Euler characteristic of the nonorientable surface $\Si$; if $\Si$ is
the connected sum of $m$ copies of $\RP^2$, then $\chi(\Si)=2-m$.

When $n=2n'$ is even, we define
$$
I_n^{\pm}(\Si)
=\{ (\nu,\tau_0(\nu)) \in I_n\setminus I_n^0\mid
\nu\in I_{n',k}, (-1)^{n'\chi(\Si)+k}= \pm 1\}.
$$
When $n$ is odd, we define
$I_n^\pm(\Si)$ to be empty sets.
Then
$$
\cA(P^{n,\pm}_\Si)=\bigcup_{\mu\in I_n^0\cup I_n^{\pm}(\Si)}\cA_\mu.
$$

By the discussion in \cite[Seciton 3.3]{HLR}, there is an inclusion
$\iota:\cA(Q^n_\Si)\hookrightarrow \cA(P^{n,+}_\Si)$, and the Morse
stratification on $\cA(Q^n_\Si)$ is given by
$$
\cA(Q^n_\Si) =\bigcup_{\mu\in I_n^0\cup I_n^+(\Si)} \cA'_\mu
$$
where $\cA'_\mu = \cA_\mu \cap \cA(Q^n_\Si)$.  Let
$$
\cG'=\Aut(Q^n_\Si)=\Map(\Si, SU(n)),\quad
\cG=\Aut(P^{n,+}_\Si) =\Map(\Si,U(n)),
$$
and let $\bN_\mu$ (resp. $\bN_\mu'$) be the normal
bundle of $\cA_\mu$ (resp. $\cA_\mu'$)
in $\cA(P^{n,+}_\Si)$ (resp. $\cA(Q^n_\Si)$).
Then there are continuous maps
$$
(\cA_\mu')_{h\cG'} \stackrel{\iota_\mu}{\hookrightarrow}
(\cA_\mu)_{h\cG'} \stackrel{q_\mu}{\to}(\cA_\mu)_{h\cG}
$$
and the vector bundle
$(\bN_\mu')_{h\cG'}$ over $(\cA'_\mu)_{h\cG'}$
is the pullback of the vector bundle $(\bN_\mu)_{h\cG}$ over
$(\cA_\mu)_{h\cG}$ under $q_\mu\circ \iota_\mu$. So if
$e_\cG(\bN_\mu)=0$ then
$e_{\cG'}(\bN_\mu')=0$.  Therefore, to prove the vanishing
of the equivariant Euler class (Theorem \ref{thm:vanishing}),
it suffices to consider the $U(n)$ case.

%----------
\subsection{Decomposition of the normal bundle}
Let $E=P^{n,k}_\tSi \times_\rho \bC^n$ be the complex vector bundle
associated to the fundamental representation $\rho:U(n)\to
GL(n,\bC)$. Then $E\to \tSi$ is a rank $n$, degree $k$ complex
vector bundle equipped with a Hermitian metric $h$, and
$\cA(P^{n,k}_\tSi)$ can be identified with $\cA(E,h)$, the space of
Hermitian connections on $(E,h)$ (i.e. connections on $E$ that are
compatible with the Hermitian metric $h$). Let $\cC(E)$ denote the
space of holomorphic structures on $E$. Then there is an isomorphism
$\cA(P^{n,k}_\tSi)\stackrel{\cong}{\to} \cC(E)$ of complex affine
spaces, given by $\nabla\mapsto \nabla^{0,1}$. Let $\cE$ denote $E$
equipped with a $(0,1)$-connection (holomorphic structure), so that
$\cE$ can be viewed as a point in $\cC(E)$ and thus a point in
$\cA(P^{n,k}_\tSi)$.

Let $\mu\in I_n^0\cup I^\pm_n(\Si)$ be as in \eqref{eqn:mu}, so that
$\cA_\mu$ is a stratum of $\cA(P^{n,\pm}_\Si)$, and
$$
\cA_\mu = \tilde{\cA}_\mu \cap \cA(P^{n,0}_\tSi)^{\tau^\pm}
$$
where $\tilde{\cA}_\mu$ is the corresponding stratum
of $\cA(P^{n,0}_\tSi)$ labeled by the same Atiyah-Bott type $\mu$.
Let $\cN_\mu$ be the critical set of $\cA_\mu$, and let
$i:\cN_\mu\hookrightarrow \cA_\mu$ be the inclusion map.
There is a gauge equivariant deformation retraction
$r: \cA_\mu \to \cN_\mu$, so
$e_\cG(\bN_\mu)=0$ if and only if $e_{\cG}(i^*\bN_\mu)=0$.
We have the following equivalences
of equivariant pairs:
\begin{eqnarray*}
&&(\cA_\mu,\cG(P^{n,\pm}_\Si))\sim
(\cN_\mu, \cG(P^{n,\pm}_\Si))\\
&& \sim \left(\cN_{0}(P^{n_0,\pm}_\Si), \cG(P^{n_0,\pm}_\Si)\right)\times
\prod_{j=1}^r \left(\cN_{ss}(P^{n_j,k_j}_\tSi),
\cG(P^{n_j,k_j}_\tSi)\right).
\end{eqnarray*}
When $\mu\in I^0_n$ so that $n_0>0$, the parity of
$P^{n_0,\pm}_\Si$ can either agree or disagree with
that of $P^{n,\pm}_\Si$.

A point in $\cN_\mu$ corresponds to a holomorphic
vector bundle $\cE$ of the form
$$
\cE= \D_1 \oplus  \cdots \oplus \D_r \oplus \D_0 \oplus
\tau_\cC(\D_r)\oplus \cdots \oplus \tau_\cC(\D_1)
$$
where $\D_j$ is a degree $k_j$, rank $n_j$ polystable vector bundle,
$\D_0$ is a degree 0, rank $n_0$ polystable vector bundle,
$\tau_\cC(\D_j)=\tau^*\overline{\D^\vee_j}$  and $\tau_\cC(\D_0)
\cong\D_0$ (see \cite[Section 3]{HLR} for more details).
%-------

Let $\tilde{\bN}_\mu$ be the normal bundle of $\tilde{\cA}_\mu$ in
$\cA(P^{n,0}_\tSi)$.
 Then the fiber of $\tilde{\bN}_\mu$ at $\cE$ is
\begin{eqnarray*}
&& (\tilde{\bN}_\mu)_\cE = H^1(\tSi,\End''(\cE))\\
&=& \bigoplus_{0<i<j} H^1\left(\tSi, \Hom(\D_i,\D_j)\right) \oplus
\bigoplus_{0<i<j} H^1\left(\tSi, \Hom(\tau_\cC(\D_j), \tau_\cC(\D_i)\right) \\
&& \oplus \bigoplus_{0<i,j}H^1\left(\tSi, \Hom(\D_i, \tau_\cC(\D_j))\right)\\
&& \oplus \bigoplus_{i>0}H^1\left(\tSi,\Hom(\D_i,\D_0)\right) \oplus
\bigoplus_{i>0}H^1\left(\tSi, \Hom(\D_0, \tau_\cC(\D_i))\right).
\end{eqnarray*}

As explained in \cite[Section 4.1]{HLR}, $\tau$ induces
conjugate linear maps of complex vector spaces:
\begin{eqnarray*}
H^1(\tSi,\Hom(\D_i,\D_j)) & \to&  H^1(\tSi, \Hom(\tau_\cC(\D_j),
\tau_\cC(\D_i))),
\textup{ and its inverse},\\
H^1(\tSi,\Hom(\D_i, \tau_\cC(\D_j)) & \to & H^1(\tSi, \Hom(\D_j, \tau_\cC(\D_i)),\\
H^1(\tSi,\Hom(\D_i, \D_0))&\to &H^1(\tSi,
\Hom(\D_0,\tau_\cC(\D_i))), \textup{ and its inverse}.
\end{eqnarray*}

Let $\bN_\mu$ be the normal bundle of $\cA_\mu$ in
$\cA(P^{n,\pm}_\Si)=\cA(P^{n,0}_\tSi)^{\tau^\pm}$. Then the
fiber of $\bN_\mu$ at $\cE$ is
\begin{equation}\label{eqn:Nmu}
\begin{aligned}
& (\bN_\mu)_\cE = H^1(\tSi,\End''(\cE))^\tau\\
\cong & \bigoplus_{0<i<j} H^1\left(\tSi, \Hom(\D_i,\D_j)\right)
\oplus
\bigoplus_{0<i<j} H^1\left(\tSi,\Hom(\D_i,\tau_\cC(\D_j))\right) \\
& \oplus
\bigoplus_{j>0}H^1\left(\tSi,\Hom(\D_j,\D_0)\right)
\oplus \bigoplus_{j>0}H^1\left(\tSi, \Hom(\D_j,
\tau_\cC(\D_j))\right)^\tau
\end{aligned}
\end{equation}

Therefore   $i^*\bN_\mu = \bN_\mu^\bC \oplus \bN_\mu^\bR$, where
\begin{eqnarray*}
(\bN_\mu^\bC)_\cE &=& \bigoplus_{0<i<j} H^1\Bigl(\tSi,\Hom(\D_i,
\D_j)\Bigr)
\oplus \bigoplus_{0<i<j} H^1\Bigl(\tSi,\Hom(\D_i, \tau_\cC(\D_j)\Bigr)\\
&& \oplus \bigoplus_{j>0} H^1\Bigl(\tSi,\Hom(\D_j, \D_0)\Bigr),\\
(\bN_\mu^\bR)_\cE &=&  \bigoplus_{j>0} H^1\Bigl(\tSi,\Hom(\D_j,
\tau_\cC(\D_j))\Bigr)^\tau.
\end{eqnarray*}
Note that $\bN_\mu^\bC$ is a complex vector bundle over $\cN_\mu$ and
$\bN_\mu^\bR$ is a real vector bundle over $\cN_\mu$.  We have
\begin{equation}\label{eqn:eee}
e_\cG(i^*\bN_\mu)=e_\cG(\bN_\mu^\bC)\cup  e_\cG(\bN_\mu^\bR).
\end{equation}
Let
$$
\lambda_\mu =\rank_\bR \bN_\mu,\quad \lambda_\mu^\bC=\rank_\bC \bN_\mu^\bC,\quad
\lambda_\mu^\bR=\rank_\bR\bN_\mu^\bR.
$$
Then
$$
\lambda_\mu =2 \lambda_\mu^\bC +\lambda_\mu ^\bR.
$$
\begin{lm}\label{lemma:ENC}
Let $K=\bQ$ or $K=\bZ_p$ ($p$ any prime). Then
$e_\cG(\bN_\mu^\bC)$ is not a zero divisor in $H^*_\cG(\cN_\mu;K)$.
\end{lm}
\begin{proof}
Let $U(1)_j$ be the center of $U(n_j)$, the group of {\em constant}
gauge transformation on $P^{n_j,k_j}_\tSi$. Let $T^r = U(1)_1 \times
\cdots U(1)_r$. Then $T^r\subset \cG(P^{n_0,\pm}_\Si)\times
\prod_{j=1}^r \cG(P^{n_j,k_j}_\tSi)$ acts trivially on
$\cN_{0}(P^{n_0,\pm}_\Si)\times \prod_{j=1}^r
\cN_{ss}(P^{n_j,k_j}_\tSi)$, and the weights of the $T^r$-action on
$\bN_\mu^\bC$ are given by
\begin{eqnarray*}
t_j t_i^{-1} && \mbox{on} \quad H^1\Bigl(\tSi,\Hom(\D_i, \D_j)\Bigr), \quad i<j,  \\
t_j t_i^{-1}&& \mbox{on} \quad H^1\Bigl(\tSi,\Hom(\D_i, \tau_\cC(\D_j)\Bigr), \quad i<j,\\
\mbox{and}\quad\quad t_j^{-1} && \mbox{on} \quad
H^1\Bigl(\tSi,\Hom(\D_j, \D_0)\Bigr),\quad j>0.
\end{eqnarray*} where $(t_1,\cdots,t_r)\in T^r$ (cf: \cite[p.569]{ym}).
So the representation of $T^r$ on the fiber of $\bN_\mu^\bC$ is primitive.
By \cite[Proposition 13.4]{ym}), $e_\cG(\bN_\mu^\bC)$ is not a zero divisor
in $H^*_\cG(\cN_\mu;K)$.
\end{proof}

By \eqref{eqn:eee} and Lemma \ref{lemma:ENC}, $e_{\cG}(i^* \bN_\mu)=0$
if and only if $e_{\cG}(\bN_\mu^{\bR})=0$.
To study $\bN_\mu^\bR$, we reduce it
to bundles over representation varieties, which we recall
in the next subsection.

\subsection{Representation varieties for Yang-Mills connections}
Let $\Si^\ell_0$ be the closed, compact, connected, orientable surface
with $\ell\geq 0$ handles. Let $\Si^\ell_1$ be the connected sum
of $\Si^\ell_0$ and $\RP^2$, and let $\Si^\ell_2$ be the connected
sum of $\Si^\ell_0$ and a Klein bottle. Any connected, closed,
nonorientable surface
is of the form $\Si^\ell_i$, where $\ell$ is a nonnegative integer and $i=1,2$.
Note that $\Si^\ell_i$ is the connected sum of $(2\ell+i)$-copies of $\RP^2$,
and that the orientable double cover of $\Si^\ell_i$ is $\Si^{\tg}_0$, where
$\tg=2\ell+i-1$.

A Yang-Mills $G$-connection on $\Si$ gives rise to a homomorphism
$\Gamma_\bR(\Si)\to G$ where $\Gamma_\bR(\Si)$ is the {\em super
central extension} introduced in \cite[Section 4.6]{HL1}.
Given $V=(\ab)\in G^{2\ell}$, define
$$
\fm(V)=\pab,\quad \fr(V)=(b_\ell,a_\ell,\ldots, b_1,a_1).
$$
In \cite{HL1}, the authors introduced the following
symmetric representation varieties of Yang-Mills connections
on the orientable double cover $\widetilde{\Si^{\ell}_i}=\Si^{\tg}_0$:
\begin{eqnarray*}
\ZymU{1}_\kn&=&
\bigl\{(V,c,V',c',-2\sqrt{-1}\pi \frac{k}{n} I_n)\mid V,V' \in U(n)^{2\ell},\ c,c'\in U(n),\\
&& \quad \fm(V) =e^{-\pi\sqrt{-1}k/n}I_n c c',\
\fm(V')=e^{\pi \sqrt{-1}k/n}I_n c' c \bigr\},\\
\ZymU{2}_\kn&=&
\bigl\{(V,d,c,V',d',c',-2\sqrt{-1}\pi \frac{k}{n} I_n)\mid V,V'\in U(n)^{2\ell} ,d,c,d', c'\in U(n),\\
&& \quad \fm(V) =e^{-\pi\sqrt{-1}k/n}I_n cd' c^{-1}d,\ \fm(V')=
e^{\pi\sqrt{-1}k/n}I_n c' d (c')^{-1}d' \bigr\}.
\end{eqnarray*}
We also have the following representation variety of
Yang-Mills connections on $\Si^{\tg}_0$:
\begin{eqnarray*}
\ymU{\tg}{0}_\kn&=& \{ (V,-2\sqrt{-1}\pi \frac{k}{n} I_n) \mid V\in
U(n)^{2\tg},\ \fm(V)=e^{-2\pi\sqrt{-1}k/n}I_n \bigr\}\\
&\cong & \cN_{ss}\bigl(P^{n,k}_{\Si^{\tg}_0}\bigr)/\cG_0\bigl(P^{n,k}_{\Si^{\tg}_0}\bigr).
\end{eqnarray*}
Note that $\ymU{0}{0}_\kn$ is empty unless $\frac{k}{n}\in \bZ$, and
$\ymU{0}{0}_{d,\ldots,d}$ consists of a point if $d\in \bZ$.

The surjective maps $\Phi^{\ell,i}:\ZymU{i}_\kn\to \ymU{2\ell+i-1}{0}_\kn$ are given by
\begin{eqnarray*}
\Phi^{\ell,1}(V,c,V',c',-2\sqrt{-1}\pi \frac{k}{n} I_n ) &=& (V, c\fr(V')c^{-1}, -2\sqrt{-1}\pi \frac{k}{n} I_n)\\
\Phi^{\ell,2}(V,d,c,V',d',c',-2\sqrt{-1}\pi \frac{k}{n} I_n ) &=&
(V,d^{-1}c \fr(V') c^{-1}d, d^{-1}, cc',-2\sqrt{-1}\pi \frac{k}{n} I_n)
\end{eqnarray*}

In particular, when $n=1$, $k\in \bZ$, we have
\begin{eqnarray*}
\ZymS{1}_k &=& \bigl\{(V,c,V',(-1)^kc^{-1},-2\sqrt{-1}\pi k)\mid
V,V\in U(1)^{2\ell},\ c\in U(1) \bigr\}\\
&\cong & U(1)^{4\ell+1},\\
\ZymS{2}_k&=& \bigl\{(V,d,c,V',(-1)^k d^{-1},c',-2\sqrt{-1}\pi k)\mid
V,V'\in U(1)^{2\ell},\ d,c,c'\in U(1) \bigr\}\\
&\cong & U(1)^{4\ell+3},\\
\ymS{\tg}{0}_k &=& \{ (V,-2\sqrt{-1}\pi k)\mid V\in U(1)^{2\tg}\} \cong U(1)^{2\tg}.
\end{eqnarray*}
The maps $\Phi^{\ell,i}:\ZymS{i}_k \cong U(1)^{4\ell+2i-1}\to \ymS{2\ell+i-1}{0}_k\cong U(1)^{4\ell+2i-2}$, $i=1,2$,
are given by
\begin{eqnarray*}
\Phi^{\ell,1}(V,c,V',(-1)^kc^{-1},-2\sqrt{-1}\pi k) &=& (V, \fr(V'), -2\sqrt{-1}\pi k)\\
\Phi^{\ell,2}(V,d,c,V',(-1)^k d^{-1},c',-2\sqrt{-1}\pi k ) &=&
(V,\fr(V'), d^{-1}, cc',-2\sqrt{-1}\pi k)
\end{eqnarray*}

\subsection{Vanishing of equivariant Euler class}

Let  $\V_{n_j,k_j} \to \cN_{ss}(P^{n_j,k_j}_\tSi)$
be the real vector bundle whose fiber over
$\D_j \in \cN_{ss}(P^{n_j,k_j})$ is
$H^1(\tSi,\Hom(\D_j,\tau_\cC(\D_j)))^\tau$.
Then  $\V_{n_j,k_j}$ is a $\cG_j$-equivariant real
vector bundle of rank $2n_jk_j + n_j^2(\tg-1)$, where
$\cG_j=\cG(P^{n_j,k_j}_\tSi)$ and  $\tg$ is the
genus of $\tSi$.

For $j=1,\ldots, r$, let
$$
\displaystyle{p_j: \cN_{ss}(P^{n_0,\pm}_\Si)\times
\prod_{i=1}^r \cN_{ss}(P^{n_i,k_i}_\tSi) }
\longrightarrow \cN_{ss}(P^{n_j,k_j}_\tSi)
$$
be the natural projection.
Under the isomorphism of equivariant pairs
$$
\bigl(\cN_\mu, \cG(P^{n,\pm}_\Si) \bigr)\cong
\bigl(\cN_{ss}(P^{n_0,\pm}_\Si),\cG(P^{n_0,\pm}_\Si\bigr) \times
\prod_{j=1}^r \bigl(\cN_{ss}(P^{n_j,k_j}_\tSi),  \cG_j\bigr)
$$
the $\cG$-equivariant vector bundle $\bN_\mu^\bR$ over $\cN_\mu$
is isomorphic to the $\prod_{j=1}^r\cG_j$-equivariant vector bundle
$\bigoplus_{j=1}^{r} p_j^* \V_{n_j,k_j}$ over $ \prod_{j=1}^r
\cN_{ss}(P^{n_j,k_j}_\tSi)$.
In other words, there is a homeomorphism of the total spaces of vector bundles
$$
(\bN_\mu^{\bR})^{h\cG} \cong \bigoplus_{j=1}^{r} p_j^* \V_{n_j,k_j}^{h\cG_j}
$$
which covers the homeomorphism of the bases
$$
\cN_\mu^{h\cG}\cong
\cN_{ss}(P^{n_0,\pm}_\Si)^{h \cG(P^{n_0,\pm}_\Si)} \times
\prod_{j=1}^r \cN_{ss}(P^{n_j,k_j}_\tSi)^{h\cG_j}.
$$
So
$$
e_\cG(\bN_\mu^\bR) =\prod_{j=1}^r e_{\cG_j}(\V_{n_j,k_j}),
$$

The $\cG_j$-equivariant vector bundle $\V_{n_j,k_j}\to \cN_{ss}(P^{n_j,k_j}_\tSi)$
descends to a $U(n_j)$-equivariant vector bundle
$V_{n_j,k_j}$ over $X_{\mathrm{YM}}^{\tg,0}(U(n_j))_{\frac{k_j}{n_j},\ldots,\frac{k_j}{n_j}}$,
and $e_{\cG_j}(\V_{n_j,k_j})$ descends to $e_{U(n_j)}(V_{n_j,k_j})$.

In the remainder of this subsection, we use rational coefficient $\bQ$.

\begin{lm}\label{lemma:ENR}
When $n=1$, $k>0$, the $U(1)$-action on $V_{1,k}$ is trivial, and
$$
e_{U(1)}(V_{1,k})= e(V_{1,k})=0
$$
\end{lm}
\begin{proof}
The $U(1)$-action is similar to that in Lemma \ref{lemma:ENC}.
%--------------proof--------

We first review some discussion in \cite[Section 6.2]{HLR}.
Given $c\in U(1)$, let $\bc=c^{-1}$ denote the complex conjugate.
Given $V=(\ab)\in U(1)^{2\ell}$, and $V'=(a'_1,b'_1,\ldots, a'_\ell,b'_\ell)$, let
$$
\bV=(\ba_1,\bb_1,\ldots, \ba_\ell,\bb_\ell),\quad
VV'=(a_1 a'_1,b_1b'_1,\ldots, a_\ell a'_\ell, b_\ell b'_\ell).
$$
The map $\LL\mapsto \Hom(\LL,\tau_\cC(\LL)) =\LL^\vee\otimes \tau^*\overline{\LL^\vee}$,
where $\LL$ is a degree $k>0$ holomorphic line bundle over
$\tSi$, induces a map  $\phi_Z:\ZymS{i}_k\longrightarrow \ZymS{i}_{-2k}$ given by
\begin{eqnarray*}
&& (V,c,V',(-1)^k \bc,-2\sqrt{-1}\pi k)\\
&\mapsto& ( \bV V',(-1)^k
\bc^2,\bV' V,(-1)^k c^2,4\sqrt{-1} \pi k),\quad i=1,\\
&& (V,d,c, V',(-1)^k \bd,c',-2\sqrt{-1}\pi k) \\
&\mapsto& (\bV V',(-1)^k \bd^2, \bc c', \bV' V,(-1)^k d^2, \bc'
c,4\sqrt{-1} \pi k),\quad i=2.
\end{eqnarray*}
It descends to a map $\phi_X:\ymS{2\ell+i-1}{0}_k\longrightarrow \ymS{2\ell+i-1}{0}_{-2k}$ given by
\begin{eqnarray*}
(V_1, V_2, -2\sqrt{-1}k) &\mapsto & (\fr(V_2) \bV_1, \fr(V_1) \bV_2, 4\sqrt{-1}\pi k),\quad i=1\\
(V_1, V_2,d,c,-2\sqrt{-1}k) &\mapsto & (\fr(V_2) \bV_1, \fr(V_1)\bV_2, (-1)^k \bd^2,1, 4\sqrt{-1}\pi k ),\quad i=2.
\end{eqnarray*}

The map $\cM\mapsto \tau^*\overline{\cM}$, where $\cM$
is a degree $-2k$ holomorphic line bundle over $\tSi$, induces
an involution $\hat{\tau}_Z:\ZymS{i}_{-2k}\to \ZymS{i}_{-2k}$ given by
\begin{eqnarray*}
&& (V,c,V',\bc, 4\sqrt{-1}\pi k)
\mapsto ( \bV', c, \bV , \bc , 4\sqrt{-1}\pi k),\quad i=1,\\
&&(V,d, c,V', \bd , c', 4\sqrt{-1}\pi k) \mapsto ( \bV' ,d, \bc' ,
\bV , \bd , \bc, 4\sqrt{-1}\pi k),\quad i=2.
\end{eqnarray*}
It descends to an involution $\hat{\tau}_X:\ymS{2\ell+i-1}{0}_{-2k} \longrightarrow  \ymS{2\ell+i-1}{0}_{-2k}$
given by
\begin{eqnarray*}
(V_1, V_2,4\sqrt{-1}\pi k) &\mapsto & (\fr(\bV_2), \fr(\bV_1), 4\sqrt{-1}\pi k),\quad i=1\\
(V_1, V_2, d, c,4\sqrt{-1}\pi k) &\mapsto & (\fr(\bV_2), \fr(\bV_1), d,\bc, 4\sqrt{-1}\pi k),\quad i=2.
\end{eqnarray*}
We have
$$
\mathrm{Im}\,\phi_Z= \ZymS{i}_{-2k}^{\hat{\tau}_Z}\cong U(1)^{2\ell+i},\quad
\mathrm{Im}\, \phi_X= \ymS{2\ell+i-1}{0}_{-2k}^{\hat{\tau}_X}\cong U(1)^{2\ell+i-1}.
$$

Let $U_k\to \ZymS{i}_{-2k}$  and $F_k\to \ymS{2\ell+i-1}{0}_{-2k}$
be the vector bundles
whose fiber over $\cM$ is $H^1(\tSi,\cM)$. Then
the involution $\hat{\tau}_Z$ (resp. $\hat{\tau}_X$) lifts to an involution on $U_k$
(resp. $F_k$):
$$
(U_k)_\cM = (F_k)_\cM= H^1(\tSi,\cM) \longrightarrow
(U_k)_{\tau^*\overline{\cM}}=(F_k)_{\tau^*\overline{\cM}}= H^1(\tSi,\tau^*\overline{\cM}).
$$
The fixed locus $U_k^{\hat{\tau}_Z}$ (resp. $F_k^{\hat{\tau}_X}$) is a real vector bundle over
$\ZymS{i}_{-2k}^{\hat{\tau}_Z}$ (resp. $\ymS{2\ell+i-1}{0}_{-2k}^{\hat{\tau}_X}$).
Let $W_k\to \ZymS{i}_k$ be the vector bundle whose fiber over $\LL$ is
$H^1(\tSi,\Hom(\LL,\tau_\cC(\LL))^\tau$. Then
$$
\phi_Z^* U_k^{\hat{\tau}_Z}=W_k,\quad
 \phi_X^* F_k^{\hat{\tau}_X} = V_{1,k},\quad \rank_\bR V_{1,k}= \rank_\bR F_k^{\hat{\tau}_X}=
\rank_\bC F_k=2k + 2\ell+i-2.
$$

The $U(1)$-action on $\Hom(\LL, \tau_\cC(\LL))$ is given by $t \cdot t^{-1}$,
and thus the weights of the $U(1)$-action on $(V_{1,k})_\LL=H^1(\tSi,
\Hom(\LL, \tau_\cC(\LL)))^\tau$ are also given by $t\cdot t^{-1}$ which is
trivial. So $e_{U(1)}(V_{1,k})=  e(V_{1,k})$.

We have
$$
\rank_\bR F_k^{\hat{\tau}_X} =  2k+ 2\ell + i -2 > 2\ell+i -1 =
\dim_\bR \ymS{2\ell+i-1}{0}_{-2k}^{\hat{\tau}_X}
$$
since $k>0$. So $e(F_k^{\hat{\tau}_X})=0$.  Therefore
$$
e(V_{1,k})= \phi_X^* e(F_k^{\hat{\tau}_X})=0.
$$
\end{proof}

\begin{proof}[Proof of Theorem \ref{thm:vanishing}]
Part (ii) follows from Lemma \ref{lemma:ENR}. For
part (i), recall
that $\ymU{0}{0}_\kn$ is empty unless
$\frac{k}{n}\in \bZ$, and $\ymU{0}{0}_{d,\cdots,d}$ consists of
a point if $d\in \bZ$. We need to prove that, for any positive integers  $n,d>0$,
$$
e_{U(n)}(V_{n,nd}) \in H^*_{U(n)}(
\ymU{0}{0}_{d,\ldots,d};\bQ)
$$
is zero. Since $Y_{n,d}:=\ymU{0}{0}_{d,\ldots,d}$ is a point,
the inclusion of the maximal torus
$T=U(1)^n\subset U(n)$ induces an injective
ring homomorphism
$$
\beta: H^*_{U(n)}(Y_{n,d};\bQ)\cong \bQ[u_1,\ldots,u_n]^{S_n}
\to H^*_T(Y_{n,d};\bQ)\cong \bQ[u_1,\ldots,u_n].
$$
So it suffices to show that
$e_T(V_{n,nd}) = \beta\left(e_{U(n)}(V_{n,nd}) \right)$ is zero.
We have
\begin{eqnarray*}
V_{n,nd} &=& H^1\Bigl(\bP^1,\Hom( \bigoplus_{i=1}^n \LL_i,
\bigoplus_{j=1}^n \tau_\cC(\LL_j) \Bigr)^\tau \\
&\cong& \bigoplus_{i<j} H^1(\bP^1, \LL_i^\vee\otimes \tau_\cC(\LL_j))
\oplus \bigoplus_{i=1}^n H^1(\bP^1, \LL_i^\vee\otimes \tau_\cC(\LL_i))^{\tau}
\end{eqnarray*}
where $\LL_i=\cO_{\bP^1}(d)$ for $i=1,\ldots,n$
and $\tau_\cC (\LL_j) = \cO_{\bP^1}(-d)$ for
$j=1,\ldots, n$.  The weights of $T$-action on
$H^1\bigl(\bP^1,\LL_i^\vee\otimes \tau_\cC(\LL_j)\bigr)$ is
$t_j t_i^{-1}$, where
$(t_1,\ldots,t_n)\in U(1)^n=T$.
Let
$$
V_\bC = \bigoplus_{i<j} H^1(\bP^1, \LL_i^{-1}\otimes \tau_\cC(\LL_j)),\quad
V_\bR = \bigoplus_{i=1}^n H^1(\bP^1, \LL_i^{-1}\otimes \tau_\cC(\LL_i))^{\tau}.
$$
Then
$$
V_{n,nd} = V_\bC\oplus V_\bR
$$
where $V_\bC$ is a complex vector space, $V_\bR$ is a real
vector space on which $T$-acts trivially, and
$$
\dim_\bR V_{n,nd}= n^2(2d-1),\quad
\dim_\bC V_\bC = \frac{n(n-1)}{2}(2d-1),\quad
\dim_\bR V_\bR = n(2d-1).
$$
We have
$$
e_T(V_{n,nd})=e_T(V_\bC) e_T(V_\bR),
$$
where
$$
e_T(V_\bC)= \pm \prod_{i<j} (u_i-u_j)^{2d-1},\quad
e_T(V_\bR)=0,
$$since $\rank_\bR V_\bR=\dim_\bR V_\bR>0=\dim_\bR Y_{n,d}$.
Therefore $e_T(V_{n,nd})=0$.
\end{proof}

\section{Equivariant Poincar\'{e} Series}\label{sec:poincare}

By P5 of Lemma \ref{lm:perfect}, the stratification is equivariantly $\bQ$-perfect if and only if
\begin{equation}\label{eqn:Pfive}
P_t^{\cG}(\cA_{ss};\bQ)=P_t^\cG(\cA;\bQ) -\sum_{\mu\in \Lambda'} t^{\lambda_\mu}P_t^{\cG}(\cA_\mu;\bQ)
\end{equation}

By A5 of Lemma \ref{lm:anti-perfect},
the stratification is equivariantly $\bQ$-antiperfect if and only if
\begin{equation}\label{eqn:Afive}
P_t^{\cG}(\cA_{ss};\bQ)=P_t^\cG(\cA;\bQ) +\sum_{\mu\in \Lambda'} t^{\lambda_\mu-1} P_t^{\cG}(\cA_\mu;\bQ).
\end{equation}

\subsection{Representation varieties for flat connections}
A flat $G$-connection on $\Si$ gives rise to a homomorphism
$\pi_1(\Si)\to G$. Recall that
\begin{eqnarray*}
\pi_1(\Si^\ell_1)
&=&\langle A_1,B_1,\ldots,A_\ell, B_\ell, C \mid \prod_{i=1}^\ell [A_i,B_i]=C^2 \rangle,\\
\pi_1(\Si^\ell_2)
&=&\langle A_1,B_1,\ldots,A_\ell, B_\ell, D,C \mid \prod_{i=1}^\ell [A_i,B_i]=C D C^{-1} D \rangle.
\end{eqnarray*}

Representation varieties of flat $U(n)$-connections and $SU(n)$-connections on $\Si^\ell_1$ and $\Si^\ell_2$ are given by
\begin{eqnarray*}
\flU{\ell}{1}&=&\{(V,c) \mid V\in U(n)^{2\ell}, c\in U(n), \fm(V)=c^2 \} \\
\flU{\ell}{1}_{\pm 1} &=& \{ (V,c)\in \flU{\ell}{1} \mid \det c =\pm 1\}\\
X_{\mathrm{flat}}^{\ell,1}(SU(n))&=& \{(V,c) \mid V\in SU(n)^{2\ell}, c\in SU(n), \fm(V)=c^2 \} \\
\flU{\ell}{2}&=&\{(V,d,c) \mid V\in U(n)^{2\ell},
d,c\in  U(n), \fm(V)= c d c^{-1} d \}\\
\flU{\ell}{2}_{\pm 1} &=& \{ (V,d,c)\in \flU{\ell}{1} \mid \det d =\pm 1\}\\
X_{\mathrm{flat}}^{\ell,2}(SU(n))&=& \{(V,d,c) \mid V\in SU(n)^{2\ell}, d,c\in SU(n), \fm(V)=c d c^{-1} d \}
\end{eqnarray*}
For $i=1, 2$,
$$
\Hom(\pi_1(\Si^{\ell}_i), U(n))_{\pm 1} = \flU{\ell}{i}_{\pm 1},\quad
\Hom(\pi_1(\Si^{\ell}_i), SU(n)) = X_{\mathrm{flat}}^{\ell,i}(SU(n)).
$$

\subsection{Rank 2 case}\label{sec:Utwo}

\begin{proof}[Proof of Theorem \ref{thm:Utwo}]
There are two possible principal $U(2)$-bundles $P^{2,+}_{\Si^\ell_i},
~P^{2,-}_{\Si^\ell_i}$ over the nonorientable surface $\Si^\ell_i$.
In  notation in Section \ref{sec:ABtypes},
\begin{eqnarray*}
&&I_2^0=\{(0,0)\}\\
&&I_2^+(\Si^\ell_1)= I_2^-(\Si^\ell_2) = \{(2r-1,1-2r)\mid r\in \bZ_{>0}\},\\
&& I_2^-(\Si^\ell_1)=I_2^+(\Si^\ell_2) = \{ (2r,-2r)\mid r\in \bZ_{>0}\}.
\end{eqnarray*}
So when $\cA=\cA(P^{2,\pm}_\Si)$, $\Lambda'= I_2^\pm(\Si)$.

Let $\tg=2\ell+i-1$ be the genus of the oriented double cover of
$\Si^\ell_i$. From \cite[Example 7.5]{HL1}, The codimension of each stratum is
$$
d_{r,-r}= 2r+\tg-1,
$$
and the equivariant Poincar\'{e} series for stratum $\mu=(r,-r)$ is
\begin{eqnarray*}
P^\cG_t\left(\cA(\Si^\ell_i)_{r,-r};\bQ\right) &=&
P^{U(2)}_t\left(X^{\ell,i}_{\mathrm{YM}}(U(2))_{r,-r};\bQ\right)
= P^{U(1)}_t\left(X^{\tg,0}_{\mathrm{YM}}(U(1))_{r};\bQ\right)\\
&=&P^{U(1)}_t(U(1)^{2\tg})= \frac{(1+t)^{2\tg} }{1-t^2}.
\end{eqnarray*}

By \cite[Theorem 2.5]{HL2},
$$
P_t^{\cG}(\cA;\bQ)= P_t(B\cG;\bQ)=\frac{(1+t)^{\tg}(1+t^3)^{\tg}}{(1-t^2)(1-t^4)}.
$$

We have
$$
\sum_{r\textup{ odd}} t^{d_{r,-r}-1} = \frac{t^{\tg}}{1-t^4},\quad
\sum_{r\textup{ even} } t^{d_{r,-r}-1 }=\frac{t^{\tg+2}}{1-t^4}.
$$
Therefore \eqref{eqn:Afive} is equivalent to the following
identities
\begin{eqnarray*}
P_t^{U(2)}\left(X^{\ell,i}_{\mathrm{flat}}(U(2))_{(-1)^i};\bQ\right)
&=& P_t(B\cG;\bQ) +\sum_{r\textup{ even} } t^{d_{r,-r}-1 }
P^\cG_t\left(\cA(\Si^\ell_i)_{r,-r};\bQ\right) \\
&=& \frac{(1+t)^{\tg}}{(1-t^2)(1-t^4)} ( (1+t^3)^{\tg} + t^{\tg+2}(1+t)^{\tg}),\\
P_t^{U(2)}\left(X^{\ell,i}_{\mathrm{flat}}(U(2))_{(-1)^{i+1}};\bQ\right)
&=& P_t(B\cG;\bQ) +\sum_{r\textup{ odd} } t^{d_{r,-r}-1 }
P^\cG_t\left(\cA(\Si^\ell_i)_{r,-r};\bQ\right) \\
&=& \frac{(1+t)^{\tg}}{(1-t^2)(1-t^4)} ( (1+t^3)^{\tg} +
t^{\tg}(1+t)^{\tg}).
\end{eqnarray*}
%Thus to get
%$P_t^{U(2)}\left(X^{\ell,i}_{\mathrm{flat}}(U(2));\bQ\right)$, we
%only needs to sum up the two components.

We now consider the principal $SU(2)$-bundles $Q^{2}_{\Si^\ell_i} \cong \Si^\ell_i \times SU(2)$ over the nonorientable surface $\Si^\ell_i$ together with the gauge group $\cG'=\Aut(Q^{2}_{\Si^\ell_i})$ action.
The set of Atiyah-Bott types is
$I_2^0 \cup I_2^+(\Si^\ell_i)$, so
$$
\Lambda'=\{(r,-r)\mid r \in \bZ_{>0},\  r = i \textup{ (mod 2) }\}.
$$
The codimension of $\cA'_{r,-r}$ in $\cA(Q^2_{\Si^\ell_i})$ is the same
as the codimension of $\cA_{r,-r}$ in $\cA(P^{2,+}_{\Si^\ell_i})$, which is
$d_{r,-r}= 2r+\tg-1$.

We now  derive the reduction formula for each stratum $\mu=(r,-r)$, $r>0$.
The corresponding representation varieties are
\begin{eqnarray*}
X^{\ell,1}_{\mathrm{YM}}(SU(2))_\mu &=&\{(V,c,X)\in SU(2)^{2\ell+1}\times C_{\mu/2}|\
V\in (SU(2)_X)^{2\ell},\\&&\quad \Ad(c)X=-X, \fm(V)=\exp(X)c^2\},\\
X^{\ell,2}_{\mathrm{YM}}(SU(2))_\mu &=&\{(V,d,c,X)\in SU(2)^{2\ell+2}\times C_{\mu/2}|\
(V,d)\in (SU(2)_X)^{2\ell+1},\\&&\quad  \Ad(c)(X)=-X,\ \fm(V)=\exp(X) cd c^{-1} d\}.
\end{eqnarray*}
where $C_{\mu/2}$ is the orbit of $X_\mu/2=-\pi\sqrt{-1}\diag(r,-r) \in \fsu(2)$ under the
Adjoint action of $SU(2)$ on $\fsu(2)$.  Let
$$
 \ep =\left( \begin{array}{cc} 0 & 1 \\  -1 & 0 \end{array}\right).
$$
Then $\Ad(\ep)(X_\mu)= -X_\mu$. Note that
$$
SU(2)_{X_\mu}=\{ \diag(u,u^{-1})\mid u\in U(1)\} \cong U(1),\quad \exp(X_\mu/2) = (-1)^r I_2.
$$
For $\mu\in \Lambda'=\{ (r,-r)\mid r>0, r = i \textup{ (mod 2) } \}$, define $V^{\ell,i}(SU(2))_\mu$ as follows:
\begin{eqnarray*}
V^{\ell,1}(SU(2))_\mu &=&\{(V,c)\in SU(2)^{2\ell+1}|\ V\in (SU(2)_{X_\mu})^{2\ell},\\
&&  \quad \Ad(c)X_\mu=-X_\mu, c^2 = -I_2 \},\\
& \stackrel{c=c'\ep}{\cong} &\{(V,c')\mid V\in (SU(2)_{X_\mu})^{2\ell}, c'\in SU(2)_{X_\mu} \} \cong U(1)^{2\ell+1}\\
V^{\ell,2}(SU(2))_\mu&=&\{(V,d,c)\in SU(2)^{2\ell+2}\mid (V,d)\in (SU(2)_{X_\mu})^{2\ell+1},\\
&& \quad  \Ad(c)(X_\mu)=-X_\mu, cd c^{-1} d = I_2\}\\
&\stackrel{c=c'\ep} {\cong} & \{ (V,d,c')\mid V\in (SU(2)_{X_\mu})^{2\ell}, d,c \in SU(2)_{X_\mu} \} \cong U(1)^{2\ell+2}
\end{eqnarray*}

By argument similar to that in \cite[Section 7]{HL1}, the following equivariant pairs are equivalent
 $$
 (X^{\ell,i}_{\mathrm{YM}}(SU(2))_\mu, SU(2))\cong
 (V^{\ell,i}(SU(2))_\mu, SU(2)_{X_\mu}) \cong
 (U(1)^{2\ell+i}, U(1))
 $$
 where $U(1)$ acts on $U(1)^{2\ell} \times U(1)^i$ by
 $$
 u\cdot (V, c) = (V, u^2 c),\quad
 u\cdot (V, d,c) = (V, d, u^2 c)
 $$

Thus, the $\cG'$-equivariant Poincar\'{e} series for stratum
$\cA'_{r,-r}$ is
$$
P^{\cG'}_t\left(\cA'_{r,-r};\bQ\right) =
P^{SU(2)}_t\left(X^{\ell,i}_{\mathrm{YM}}(SU(2))_{r,-r};\bQ\right) =
P_t(U(1)^{\tg};\bQ)=(1+t)^{\tg},\quad \tg=2\ell+i-1.
$$
By \cite[Theorem 2.5]{HL2},
$$P_t^{\cG'}(\cA(Q^2_{\Si^\ell_i});\bQ)=P_t(B\cG';\bQ)=\frac{(1+t^3)^{\tg}}{1-t^4}.$$

Therefore \eqref{eqn:Afive} is equivalent to
the following identities
\begin{eqnarray*}
P_t^{SU(2)}(X^{\ell,1}_{\mathrm{flat}}(SU(2));\bQ)&=&P_t(B\cG';\bQ)
+\sum_{r\textup{ odd}} t^{d_{r,-r}-1} (1+t)^{\tg}
=\frac{(1+t^3)^{\tg} + t^{\tg}(1+t)^{\tg}}{1-t^4}\\
P_t^{SU(2)}(X^{\ell,2}_{\mathrm{flat}}(SU(2));\bQ)&=&P_t(B\cG';\bQ)
+\sum_{r\textup{ even}} t^{d_{r,-r}-1} (1+t)^{\tg}
= \frac{(1+t^3)^{\tg} + t^{\tg+2}(1+t)^{\tg}}{1-t^4}.
\end{eqnarray*}
\end{proof}

\subsection{Rank 3 case}\label{sec:Uthree}
\begin{proof}[Proof of Theorem \ref{thm:Uthree}]
There are two possible principal $U(3)$-bundles $P^{3,+}_{\Si^\ell_i},
~P^{3,-}_{\Si^\ell_i}$ over the nonorientable surface $\Si^\ell_i$.
In the notation of Section \ref{sec:ABtypes},
$$
I_3=I_3^0 =\{(0,0,0)\} \cup \{(r,0,-r)\mid r\in \bZ_{>0}\}
$$
So when $\cA=\cA(P^{3,\pm}_\Si)$, $\Lambda'=\{(r,0,-r)\}
\mid r\in \bZ_{>0}\}$.

Let $\tg=2\ell+i-1$ be the genus of the oriented double cover of
$\Si^\ell_i$. From \cite[Example
7.6]{HL1}, the codimension of each stratum is
$$
d_{r,0,-r}= 4r+3(\tg-1),
$$
and the equivariant Poincar\'{e} series for stratum $\mu=(r,0,-r)$
is
$$
P^\cG_t\left(\cA(\Si^\ell_i)_{r,0,-r};\bQ\right) =
P^{U(3)}_t\left(X^{\ell,i}_{\mathrm{YM}}(U(3))_{r,0,-r};\bQ\right) =
P^{U(1)\times U(1)}_t\left(U(1)^{3\tg}\right) =\frac{(1+t)^{3\tg}
}{(1-t^2)^2}.
$$

By \cite[Theorem 2.5]{HL2},
$$
P_t^\cG(\cA;\bQ)=P_t(B\cG;\bQ)=\frac{(1+t)^{\tg}(1+t^3)^{\tg}(1+t^5)^{\tg}}{(1-t^2)(1-t^4)(1-t^6)}.
$$
%We have
%$$
%\sum_{r>0} t^{d_{r,0,-r}-1} = \frac{t^{3\tg}}{1-t^4}
%$$

Therefore \eqref{eqn:Afive} is equivalent to
the following identity
\begin{eqnarray*}
\lefteqn{ P_t^{U(3)}\left(X^{\ell,i}_{\mathrm{flat}}(U(3))_{\pm
1};\bQ\right)=
P_t(B\cG;\bQ) +\sum_{r>0} t^{d_{r,0,-r}-1} P^\cG_t\left(\cA(\Si^\ell_i)_{r,0,-r};\bQ\right) }\\
&&= \frac{(1+t)^{\tg}}{(1-t^2)(1-t^4)(1-t^6)} (
(1+t^3)^{\tg}(1+t^5)^{\tg} + t^{3\tg}(1+t)^{2\tg} (1+t^2+t^4)).
\end{eqnarray*}

We now consider the principal $SU(3)$-bundles $Q^{3}_{\Si^\ell_i} \cong \Si^\ell_i \times SU(3)$ over the nonorientable surface $\Si^\ell_i$ together with the gauge group $\cG'=\Aut(Q^{3}_{\Si^\ell_i})$ action.
The set of Atiyah-Bott types is
$I_3^0$, so  $\Lambda'=\{(r,0,-r)\mid r\in \bZ_{>0}\}$.
The codimension of $\cA'_{r,0,-r}$ in $\cA(Q^3_{\Si^\ell_i})$ is the same
as the codimension of $\cA_{r,0,-r}$ in $\cA(P^{3,+}_{\Si^\ell_i})$, which is
$d_{r,0,-r}= 4r+3(\tg-1)$.

We now  derive the reduction formula for each stratum $\mu=(r,0,-r)$.
The corresponding representation varieties are
\begin{eqnarray*}
X^{\ell,1}_{\mathrm{YM}}(SU(3))_{r,0,-r}&=&\{(V,c,X)\in SU(3)^{2\ell+1}\times C_{\mu/2}|\
V\in (SU(3)_X)^{2\ell},\\&&\quad Ad(c)X=-X, \fm(V)=\exp(X)c^2\},\\
X^{\ell,2}_{\mathrm{YM}}(SU(3))_{r,0,-r}&=&\{(V,d,c,X)\in SU(3)^{2\ell+2}\times C_{\mu/2}|\
(V,d)\in (SU(3)_X)^{2\ell+1},\\&&\quad  \Ad(c)(X)=-X,\ \fm(V)=\exp(X) cd c^{-1} d\}.
\end{eqnarray*}
where $C_{\mu/2}$ is the orbit of $X_\mu/2=-\pi\sqrt{-1}\diag(r,0,-r) \in \fsu(3)$
under the Adjoint action of $SU(3)$ on $\fsu(3)$. Let
$$
\ep=\left(\begin{array}{ccc}
0&0 &  1 \\
0&-1 & 0\\
1&0 & 0
 \end{array}\right)\in SU(3).
$$
Then $\Ad(\ep)X_\mu =-X_\mu$. Note that
$$
SU(3)_{X_\mu}= \{ \diag(u_1, u_2, u_3)\mid u_1, u_2, u_3 \in U(1), u_1 u_2 u_3 =1\} \cong U(1)\times U(1),
$$
$$
\exp(X_\mu/2)=\diag( (-1)^r, 1, (-1)^r).
$$

Given $\mu\in \Lambda'=\{ (r,0,-r)\mid r\in \bZ_{>0}\}$, define $V^{\ell,i}(SU(3))_\mu$ as follows:
\begin{eqnarray*}
V^{\ell,1}(SU(3))_\mu&=&\{(V,c')\in (SU(3)_{X_\mu})^{2\ell+1}|\ \fm(V)=\exp(X_\mu/2)(\epsilon c')^2\}\\
V^{\ell,2}(SU(3))_\mu&=&\{(V,d,c')\in (SU(3)_{X_\mu})^{2\ell+2}|\ \fm(V) =\exp(X_\mu/2) \epsilon c'd (\epsilon  c')^{-1}d\}.
\end{eqnarray*}
By argument similar to that in \cite[Section 7]{HL1}, the following equivariant pairs are equivalent:
\begin{eqnarray*}
&& (X^{\ell,i}_{\mathrm{YM}}(SU(3))_\mu, SU(3))\cong
(V^{\ell,i}(SU(3))_\mu, SU(3)_{X_\mu}) \\
&& \cong (Z^{\ell,i}_{\mathrm{YM}}(U(1))_r, U(1)\times U(1))
\cong (X^{\tg,0}_{\mathrm{YM}}(U(1))_r, U(1))
\end{eqnarray*}
where $Z^{\ell,i}_{\mathrm{YM}}(U(1))$ is the
symmetric representation varieties defined in \cite[Section 4.4]{HL1}.
Thus, the $\cG'$-equivariant Poincar\'{e} series for stratum
$\cA'_{r,0,-r}$ is
$$
P^{\cG'}_t\left(\cA'_{r,0,-r};\bQ\right) =
P^{SU(3)}_t\left(X^{\ell,i}_{\mathrm{YM}}(SU(3))_{r,0,-r};\bQ\right) =
P_t^{U(1)}(X^{\tg,0}_{\mathrm{YM}}(U(1))_r)=\frac{(1+t)^{2\tg}}{1-t^2}.
$$

By \cite[Theorem 2.5]{HL2},
$$P_t^{\cG'}(\cA(Q^3_{\Si^\ell_i});\bQ)=P_t(B\cG';\bQ)=\frac{(1+t^3)^{\tg}(1+t^5)^{\tg}}{(1-t^4)(1-t^6)}.$$
Therefore \eqref{eqn:Afive} is equivalent to
the following identity
\begin{eqnarray*}
P_t^{SU(3)}(X^{\ell,i}_{\mathrm{flat}}(SU(3));\bQ)&=&P_t(B\cG';\bQ)
+\sum_{r>0}t^{4r+3(\tg-1)-1}\frac{(1+t)^{2\tg}}{1-t^2}\\&
=&\frac{(1+t^3)^{\tg}(1+t^5)^{\tg}}{(1-t^4)(1-t^6)}+\frac{(1+t)^{2\tg}t^{3\tg}}{(1-t^2)(1-t^4)}
\end{eqnarray*}
\end{proof}

\end{document}